\documentclass{amsart}
\usepackage{amsmath,amsfonts,amsthm}
\usepackage{amssymb,latexsym}
\usepackage{graphics}
\usepackage[colorlinks]{hyperref}

\usepackage[all]{xypic}
\usepackage{amssymb}

\theoremstyle{plain}
\theoremstyle{definition}
\newtheorem{theorem}{Theorem}[section]
\newtheorem{lemma}[theorem]{Lemma}

\newtheorem{note}[theorem]{Note}

\newtheorem{convention}[theorem]{Convention}
\newtheorem{remark}[theorem]{Remark}
\theoremstyle{remark}

\numberwithin{equation}{section}
\newcommand{\SP}{\: \: \: \: \:}

\title[Between Devaney chaotic and Li--Yorke chaotic generalized shifts]{On a class between Devaney chaotic and Li--Yorke chaotic generalized shift dynamical systems}
\author[F. Ayatollah Zadeh Shirazi, F. Ebrahimifar, M. Hagh Jooyan, A. Hosseini]{Fatemah Ayatollah Zadeh Shirazi, Fatemeh Ebrahimifar, \\
Maryam Hagh Jooyan, Arezoo Hosseini}
\begin{document}
\begin{abstract} \noindent
In the following text, 
for finite discrete $X$ with at least two elements, nonempty countable $\Gamma$, and $\varphi:\Gamma\to\Gamma$
we prove the generalized shift dynamical system $(X^\Gamma,\sigma_\varphi)$ is densely chaotic  if and only if
$\varphi:\Gamma\to\Gamma$ does not have any (quasi--)periodic point.
Hence the class of all densely chaotic generalized shifts on $X^\Gamma$ is intermediate between
the class of all Devaney chaotic generalized shifts on $X^\Gamma$ and
the class of all Li--Yorke chaotic generalized shifts on $X^\Gamma$. 
In addition, these inclusions are proper for infinite countable $\Gamma$.
\\
Moreover we prove $(X^\Gamma,\sigma_\varphi)$ is 
Li--Yorke sensitive (resp. sensitive, strongly sensitive, asymptotic sensitive, 
syndetically sensitive, cofinitely sensitive, multi--sensitive, ergodically sensitive, spatiotemporally chaotic,
Li--Yorke chaotic) if and only if $\varphi:\Gamma\to\Gamma$ has at least one non--quasi--periodic point.
\end{abstract}
\maketitle
\noindent {\small {\bf 2020 Mathematics Subject Classification:}  37B99 \\
{\bf Keywords:}} Asymptotic sensitive; Densely chaotic; Li-Yorke sensitive;
Sensitive; Spatiotemporally chaotic; Strongly sensitive; Generalized shift.

\section{Introduction}\label{introduction}
\noindent One of the first concepts is introduced in the intermediate mathematical level,
is the concept of a self-map $\varphi:\Gamma\to\Gamma$.
We recall that a point
$a\in\Gamma$ is a {\it periodic point} of $\varphi:\Gamma\to\Gamma$ if there exists $n\geq1$
with $\varphi^n(a)=a$ and it is a {\it quasi--periodic point} of
$\varphi:\Gamma\to\Gamma$ if there exist
$n>m\geq1$ with $\varphi^n(a)=\varphi^m(a)$. We denote the collection of all
periodic (resp. non--quasi--periodic) points of $\varphi:\Gamma\to\Gamma$ with
$Per(\varphi)$ (resp. $W(\varphi)$).
\\
One may find with an elementary approach that the
map $\varphi:\Gamma\to\Gamma$ does not have any periodic point
if and only if $W(\varphi)=\Gamma$ (note that
if $\alpha\in \Gamma\setminus W(\varphi)$,
then it is a quasi--periodic point
and there exist $n>m\geq1$ with $\varphi^m(\alpha)=\varphi^n(\alpha)$
by $\varphi^{n-m}(\varphi^m(\alpha))=\varphi^n(\alpha)$ we have
$\varphi^{n-m}(\varphi^m(\alpha))=\varphi^m(\alpha)$
then $\varphi^m(\alpha)\in Per(\varphi)$,
and $Per(\varphi)\neq\varnothing$).
\\
Now for nonempty set $\Gamma$ and $\varphi:\Gamma\to\Gamma$ we have the following diagram:
{\small
\[\xymatrix{\varphi{\rm \: is \: one \: to \: one \: and\:}Per(\varphi)=\varnothing\ar@{=>}[r] &
W(\varphi)=\Gamma\ar@{=>}[r] & W(\varphi)\neq\varnothing}\]}
For ``suitable'' generalized shift dynamical system $(X^\Gamma,\sigma_\varphi)$
which we deal in this paper,
the first statement in the above diagram
``$\varphi$ is one to one and $Per(\varphi)=\varnothing$'' is equivalent to
``$(X^\Gamma,\sigma_\varphi)$ is Devaney chaotic'' \cite[Theorem 2.13]{dev} and the last statement
``$W(\varphi)\neq\varnothing$'' is equivalent to
``$(X^\Gamma,\sigma_\varphi)$ is Li-Yorke chaotic'' \cite[Theorem 3.3]{li}.
We return to this diagram in Section~4 and show that
the second statement ``$W(\varphi)=\Gamma$'' is equivalent to
``$(X^\Gamma,\sigma_\varphi)$ is densely chaotic'', moreover the implications
of the diagram are not  reversible for infinite $\Gamma$.
\\
Moreover ``sensitive to initial conditions'' or sometimes known as ``butterfly effect'' may is the first concept in
sensitivity approach for a large group of the mathematicians (see \cite{sysdevaney}). However, regarding
different dynamical points of view nowadays we can find various types of sensitivity, old ones and new ones
in metric dynamical systems like mean sensitive,
Li--Yorke sensitivity,  strongly sensitive, ergodically sensitive, multi--sensitive, cofinitely sensitive ... (see e.g.,~\cite{filip, kolyada, chin, was, sharma, tw}), or even in general dynamical systems \cite{fedeli}.
\\
In last section of this text a collection of different types of sensitivities in the category of compact metric generalized shift dynamical systems are compared.
\\
In this text by $\mathbb N$ we mean the set of positive integers $\{1,2,\ldots\}$. Also 
for finite set $A$, $|A|$ denotes the cardinality of $A$
\subsection*{What is a generalized shift?}
\noindent By a {\it dynamical system} (or briefly {\it system})
$(Z,f)$ we mean a topological space $Z$ and continuous map
$f:Z\to Z$. One of the most famous dynamical systems is one-sided
shift dynamical system $\sigma:\{1,\ldots,k\}^{\mathbb
N}\to\{1,\ldots,k\}^{\mathbb N}$ with $\sigma((x_n)_{n\in{\mathbb
N}})=(x_{n+1})_{n\in{\mathbb N}}$ (for $(x_n)_{n\in{\mathbb
N}}\in\{1,\ldots,k\}^{\mathbb N}$). Studying dynamical and
non-dynamical properties of one-sided shift has been considered
by several authors, however the reader may find interesting ideas in \cite{new-york}
too.
\\
For nonempty arbitrary set $X$ with at least two elements and
nonempty set $\Gamma$, we call $\sigma_\varphi:X^\Gamma\to X^\Gamma$
with $\sigma_\varphi((x_\alpha)_{\alpha\in\Gamma})=(x_{\varphi(\alpha)})_{\alpha\in\Gamma}$
(for $(x_\alpha)_{\alpha\in\Gamma}\in X^\Gamma$) a {\it generalized shift}. If $X$ has a
topological structure and $X^\Gamma$ equipped with product (pointwise
convergence) topology, then it is evident that $\sigma_\varphi:X^\Gamma\to X^\Gamma$
is continuous, so we may consider the dynamical system $(X^\Gamma,\sigma_\varphi)$.
Moreover $X^\Gamma$ is a compact metrizable space if and only if $X$ is compact
metrizable and $\Gamma$ is countable. Generalized shift has been introduced for the first
time in 
\cite{note}, which has been followed by studying several properties
of generalized shifts like, topological entropy~\cite{set},
Devaney chaos \cite{dev} and Li-Yorke chaos \cite{li}.
\section{Preliminaries in dynamical systems}
\noindent In the dynamical system $(Z,f)$ with compact metric phase space $(Z,d)$,
we say $x,y\in Z$ are {\it scrambled} or $(x,y)$ is a
{\it scrambled pair} of $(Z,f)$ if
\[{\displaystyle\liminf_{n\to\infty}
d(f^n(x),f^n(y))}=0\:\:{\rm and}\:\:{\displaystyle\limsup_{n\to\infty}
d(f^n(x),f^n(y))}>0\:.\]
We denote the collection of all scarmbled pairs of $(Z,f)$ with $S((Z,d),f)$
or briefly $S(Z,f)$.
We call a subset $A$ of $Z$ with at least two elements an {\it scrambled set} if
every distinct pairs of elements of $A$ is an scrambled pair, i.e.,
$A\times A\subseteq S(Z,f)\cup\Delta_Z$, where
$\Delta_Z=\{(z,z):z\in Z\}$.
Also for $\varepsilon>0$ denote
$\{(x,y)\in
S(Z,f):
{\displaystyle\limsup_{n\to\infty}d(f^n(x),f^n(y))}>\varepsilon\}$
by  $S_\varepsilon((Z,d),f)$
or briefly $S_\varepsilon(Z,f)$.
\begin{note}\label{taha}
Suppose $d$ and $d'$ are two compatible metrics on
compact metrizable space $Z$, then for every $\varepsilon>0$
exists $\delta>0$ such that
\[\forall x,y\in Z\:(d'(x,y)<\delta\Rightarrow
d(x,y)<\varepsilon).\]
Now suppose $\varepsilon,\delta>0$ satisfy the above statement,
$f:Z\to Z$ is continuous, and
$(x,y)\in S_\varepsilon((Z,d),f)$, then
${\displaystyle\liminf_{n\to\infty}d(f^n(x),f^n(y))}=0$ thus
there exists subsequence
$\{d(f^{n_k}(x),f^{n_k}(y))\}_{k\geq1}$
such that
\[{\displaystyle\lim_{k\to\infty}d(f^{n_k}(x),f^{n_k}(y))=
\liminf_{n\to\infty}d(f^n(x),f^n(y))}=0\:.\]
Since $Z\times Z$ is a compact
metrizable space, there exists subsequence
\linebreak $\{(f^{n_{k_l}}(x),f^{n_{k_l}}(y))\}_{l\geq1}$
of $\{(f^{n_k}(x),f^{n_k}(y))\}_{k\geq1}$ converging to a point of $Z\times Z$
like $(z,w)$, hence
$d(z,w)={\displaystyle\lim_{l\to\infty}d(f^{n_{k_l}}(x),f^{n_{k_l}}(y))
=\lim_{k\to\infty}d(f^{n_k}(x),f^{n_k}(y))}=0$ and $z=w$, therefore
${\displaystyle\lim_{l\to\infty}d'(f^{n_{k_l}}(x),f^{n_{k_l}}(y))=d'(z,w)=0}$,
and
\[{\displaystyle\liminf_{n\to\infty}d'(f^n(x),f^n(y))}=0\:.\]
Moreover,
${\displaystyle\limsup_{n\to\infty}d(f^n(x),f^n(y))}>\varepsilon$, hence
there exists subsequence 
\linebreak $\{d(f^{m_t}(x),f^{m_t}(y))\}_{t\geq1}$
such that
\[{\displaystyle\lim_{t\to\infty}d(f^{m_t}(x),f^{m_t}(y))=
\limsup_{n\to\infty}d(f^n(x),f^n(y))}>\varepsilon\:.\]
Again, since $Z\times Z$ is a compact
metrizable space, there exists subsequence
\linebreak $\{(f^{m_{t_l}}(x),f^{m_{t_l}}(y))\}_{l\geq1}$
of $\{(f^{m_t}(x),f^{m_t}(y))\}_{t\geq1}$ converging to a point of $Z\times Z$
like $(u,v)$, hence $d(u,v)={\displaystyle\lim_{l\to\infty}d(f^{m_{t_l}}(x),f^{m_{t_l}}(y))
=\limsup_{n\to\infty}d(f^n(x),f^n(y))}>\varepsilon$ therefore $d'(u,v)\geq\delta$.
Thus
\[{\displaystyle\limsup_{n\to\infty}d'(f^n(x),f^n(y))}\geq
{\displaystyle\lim_{l\to\infty}d'(f^{m_{t_l}}(x),f^{m_{t_l}}(y))}=d'(u,v)\geq
\delta>\delta/2\:,\]
which shows $(x,y)\in S_{\delta/2}((Z,d'),f)$
So we have, $S_\varepsilon((Z,d),f)\subseteq S_{\delta/2}((Z,d'),f)$.
\\
Using the above argument we have:
\[\forall\varepsilon>0\:\exists\mu>0\:(S_{\varepsilon}((Z,d),f)\subseteq
S_{\mu}((Z,d'),f))\]
and $S((Z,d),f)=S((Z,d'),f)$,
i.e. the definition
of an scrambled pair is independent of
chosen compatible metric on $Z$ (for more details see \cite{li}).
\end{note}
\noindent We recall that for $\varepsilon>0$ the dynamical system $(Z,f)$, 
with compact metric phase space
$(Z,d)$,  is:
\begin{itemize}
\item {\it Li-Yorke chaotic}, if $Z$ has an uncountable scrambled subset;
\item {\it Li-Yorke sensitive}, if there exists $\kappa>0$ such that
for every $x\in Z$ and open neighbourhood $U$ of $x$ there exists
$y\in U$ with $(x,y)\in S_\kappa(Z,f)$~\cite{kolyada};
\item {\it densely $\varepsilon-$chaotic}, if
$S_\varepsilon(Z,f)$ is a dense subset of $Z\times Z$~\cite{mor};
\item {\it spatiotemporally chaotic}, if
for every $x\in Z$ and open neighbourhood $U$ of $x$ there exists
$y\in U$ such that $x,y$ are scrambled~\cite{nat};
\item {\it densely chaotic}, if $S(Z,f)$ is a dense subset of $Z\times Z$ \cite{mor};
\item {\it topological transitive}, if for all opene (nonempty and open) subsets $U,V$ of
$Z$ there exists $n\geq1$ with $U\cap f^n(V)\neq\varnothing$~\cite{mai};
\item {\it Devaney chaotic}, if it is topological transitive, $Per(f)$ is dense
in $Z$ (i.e., $(Z,f)$ has {\it dense periodic points}),
and it is sensitive dependence to initial conditions
(by~\cite{banks} sensitivity dependence to initial conditions
is redundant).
\end{itemize}
\begin{convention}\label{taha10}:
In the following text suppose $X$ is a finite discrete space with
at least two elements, $\Gamma$ is a nonempty countable set, and
$\varphi:\Gamma\to\Gamma$ is a self-map.
\\
Suppose $\Gamma=\{\beta_1,\beta_2,\ldots\}$ equip $X^\Gamma$ with metric:
\[D((x_\alpha)_{\alpha\in\Gamma},(y_\alpha)_{\alpha\in\Gamma})
={\displaystyle\sum_{n\geq1}\frac{\delta(x_{\beta_n},y_{\beta_n})}{2^n}}\SP\SP\SP
((x_\alpha)_{\alpha\in\Gamma},(y_\alpha)_{\alpha\in\Gamma}\in X^\Gamma)\]
where 
\[\delta(a,b)=\left\{\begin{array}{lc} 0 & a=b\:, \\ 1 & a\neq b \:. \end{array}\right.\]
Then $D$ is a compatible metric with product topology of $X^\Gamma$.
\end{convention}
\section{Densely chaotic generalized shift dynamical systems}
\noindent In this section we prove that the system
$(X^\Gamma,\sigma_\varphi)$ is densely chaotic if and only if
$\varphi:\Gamma\to\Gamma$ does not have any periodic point, i.e.
$W(\varphi)=\Gamma$.
\begin{lemma}\label{taha40}
If $\varphi:\Gamma\to\Gamma$ does not have any periodic point, then for all finite
nonempty subsets $A,B$ of $\Gamma$, $\{n\in\mathbb{Z}:
\varphi^n(A)\cap B\neq\varnothing\}$ has at most ${\rm card}(A){\rm card}(B)$
elements and it is finite.
\end{lemma}
\begin{proof}
Since $\varphi:\Gamma\to\Gamma$ does not have any periodic point, it
does not have any quasi--periodic point too.
We prove for $\alpha,\beta\in\Gamma$,
$K=\{n\in\mathbb{Z}:\beta\in\varphi^n(\{
\alpha\})\}$ is void or singleton. Otherwise there exists distinct $n,m\in K$.
Suppose $n<m$, we have the following cases:
\begin{itemize}
\item $0\leq n<m$. In this case $\varphi^{n+1}(\alpha)=\varphi(\beta)=
\varphi^{m+1}(\alpha)$ and $\alpha$ is a quasi--periodic point of $\varphi$.
\item $n<0\leq m$. In this case $\varphi^{-n}(\beta)=\alpha$ and $\beta=
\varphi^m(\alpha)$ which leads to $\varphi^{-n+m+1}(\alpha)=
\varphi^{-n+1}(\beta)=\varphi(\alpha)$
and $\alpha$ is a quasi--periodic point of $\varphi$.
\item $n<m<0$. In this case $\varphi^{-n}(\beta)=\alpha=\varphi^{-m}(\beta)$
and $\beta$ is a quasi--periodic point of $\varphi$.
\end{itemize}
Using the above cases $\varphi:\Gamma\to\Gamma$ has a quasi--periodic point,
which is a contradiction. Hence $K$ has at most one element,
which leads to the fact that
\[\{n\in\mathbb{Z}:
\varphi^n(A)\cap B\neq\varnothing\}=
\bigcup\{\{n\in\mathbb{Z}:\beta\in\varphi^n(\{\alpha\})\}:(\alpha,\beta)
\in A\times B\}\]
has at most ${\rm card}(A){\rm card}(B)$
elements.
\end{proof}
\begin{lemma}\label{taha50}
Suppose $\varphi:\Gamma\to\Gamma$ does not have any periodic point,
$((x_\alpha)_{\alpha\in\Gamma},(y_\alpha)_{\alpha\in\Gamma})\in
S(X^\Gamma,\sigma_\varphi)$,
$((z_\alpha)_{\alpha\in\Gamma},(w_\alpha)_{\alpha\in\Gamma})\in X^\Gamma
\times X^\Gamma$ and there exist $\psi_1,\ldots,\psi_n\in\Gamma$ such that
for all $\alpha\in\Gamma\setminus\{\psi_1,\ldots,\psi_n\}$ we have
both $x_\alpha=z_\alpha$ and $y_\alpha=w_\alpha$. Then
$((z_\alpha)_{\alpha\in\Gamma},(w_\alpha)_{\alpha\in\Gamma})\in
S(X^\Gamma,\sigma_\varphi)$ and:
\[{\displaystyle\limsup_{t\to\infty}D(\sigma_\varphi^t(
(x_\alpha)_{\alpha\in\Gamma}),\sigma_\varphi^t(
(y_\alpha)_{\alpha\in\Gamma}))}={\displaystyle\limsup_{t\to\infty}
D(\sigma_\varphi^t((z_\alpha)_{\alpha\in\Gamma}),\sigma_\varphi^t(
(w_\alpha)_{\alpha\in\Gamma}))}\:.\]
\end{lemma}
\begin{proof}
Given $\varepsilon>0$ there exists $N\geq1$ with
${\displaystyle\sum_{i\geq N}\frac1{2^i}}<\frac{\varepsilon}2$.
Using Lemma~\ref{taha40} there exists $K\geq1$ with
\begin{center}
$\varphi^t(\{\beta_1,\ldots,\beta_N\})\cap\{\psi_1,\ldots,\psi_n\}=\varnothing$
for all $t\geq K$.
\end{center}
Thus $x_{\varphi^t(\beta_i)}=z_{\varphi^t(\beta_i)}$ and
$y_{\varphi^t(\beta_i)}=w_{\varphi^t(\beta_i)}$ for all $1\leq i\leq N$ and
$t\geq K$. Hence for all $t\geq K$ we have:


 $\left|D(\sigma_\varphi^t((x_\alpha)_{\alpha\in\Gamma}),\sigma_\varphi^t(
(y_\alpha)_{\alpha\in\Gamma}))-
D(\sigma_\varphi^t((z_\alpha)_{\alpha\in\Gamma}),\sigma_\varphi^t(
(w_\alpha)_{\alpha\in\Gamma}))\right|$
\begin{align*}
& =  \left|{\displaystyle\sum_{i\geq1}\frac{\delta(x_{\varphi^t(\beta_i)},
y_{\varphi^t(\beta_i)})}{2^i}}-
{\displaystyle\sum_{i\geq1}\frac{\delta(z_{\varphi^t(\beta_i)},
w_{\varphi^t(\beta_i)})}{2^i}}\right| \\
& \leq  \left|{\displaystyle\sum_{1\leq i\leq N}\left(
\frac{\delta(x_{\varphi^t(\beta_i)},
y_{\varphi^t(\beta_i)})-\delta(z_{\varphi^t(\beta_i)},
w_{\varphi^t(\beta_i)})}{2^i}\right)}\right|+ \\
&  \SP\SP\SP\SP\SP
\left|{\displaystyle\sum_{i>N}\left(
\frac{\delta(x_{\varphi^t(\beta_i)},
y_{\varphi^t(\beta_i)})-\delta(z_{\varphi^t(\beta_i)},
w_{\varphi^t(\beta_i)})}{2^i}\right)}\right| \\
& =  \left|{\displaystyle\sum_{1\leq i\leq N}\left(
\frac{\delta(x_{\varphi^t(\beta_i)},
y_{\varphi^t(\beta_i)})-\delta(x_{\varphi^t(\beta_i)},
y_{\varphi^t(\beta_i)})}{2^i}\right)}\right|+ \\
&  \SP\SP\SP\SP\SP
\left|{\displaystyle\sum_{i>N}\left(
\frac{\delta(x_{\varphi^t(\beta_i)},
y_{\varphi^t(\beta_i)})-\delta(z_{\varphi^t(\beta_i)},
w_{\varphi^t(\beta_i)})}{2^i}\right)}\right| \\
& =  \left|{\displaystyle\sum_{i>N}\left(
\frac{\delta(x_{\varphi^t(\beta_i)},
y_{\varphi^t(\beta_i)})-\delta(z_{\varphi^t(\beta_i)},
w_{\varphi^t(\beta_i)})}{2^i}\right)}\right| \\
& \leq  {\displaystyle\sum_{i>N}\left(
\frac{\delta(x_{\varphi^t(\beta_i)},
y_{\varphi^t(\beta_i)})+\delta(z_{\varphi^t(\beta_i)},
w_{\varphi^t(\beta_i)})}{2^i}\right)}\leq {\displaystyle\sum_{i>N}
\frac{2}{2^i}}<\varepsilon \: . 
\end{align*}
Using
{\small \[\left|D(\sigma_\varphi^t((x_\alpha)_{\alpha\in\Gamma}),\sigma_\varphi^t(
(y_\alpha)_{\alpha\in\Gamma}))-
D(\sigma_\varphi^t((z_\alpha)_{\alpha\in\Gamma}),\sigma_\varphi^t(
(w_\alpha)_{\alpha\in\Gamma}))\right|<\varepsilon\SP (\forall t\geq K)\:,\]}
we have:
{\small \[\left|{\displaystyle\limsup_{t\to\infty}D(\sigma_\varphi^t(
(x_\alpha)_{\alpha\in\Gamma}),\sigma_\varphi^t(
(y_\alpha)_{\alpha\in\Gamma}))}-{\displaystyle\limsup_{t\to\infty}
D(\sigma_\varphi^t((z_\alpha)_{\alpha\in\Gamma}),\sigma_\varphi^t(
(w_\alpha)_{\alpha\in\Gamma}))}\right|\leq\varepsilon \:,\] }
and
{\small \[\left|{\displaystyle\liminf_{t\to\infty}D(\sigma_\varphi^t(
(x_\alpha)_{\alpha\in\Gamma}),\sigma_\varphi^t(
(y_\alpha)_{\alpha\in\Gamma}))}-{\displaystyle\liminf_{t\to\infty}
D(\sigma_\varphi^t((z_\alpha)_{\alpha\in\Gamma}),\sigma_\varphi^t(
(w_\alpha)_{\alpha\in\Gamma}))}\right|\leq\varepsilon\:.\] }
Note to the fact that $\varepsilon>0$ is arbitrary, we have:
\[{\displaystyle\limsup_{t\to\infty}D(\sigma_\varphi^t(
(x_\alpha)_{\alpha\in\Gamma}),\sigma_\varphi^t(
(y_\alpha)_{\alpha\in\Gamma}))}={\displaystyle\limsup_{t\to\infty}
D(\sigma_\varphi^t((z_\alpha)_{\alpha\in\Gamma}),\sigma_\varphi^t(
(w_\alpha)_{\alpha\in\Gamma}))}\:,\]
and
\[{\displaystyle\liminf_{t\to\infty}D(\sigma_\varphi^t(
(x_\alpha)_{\alpha\in\Gamma}),\sigma_\varphi^t(
(y_\alpha)_{\alpha\in\Gamma}))}={\displaystyle\liminf_{t\to\infty}
D(\sigma_\varphi^t((z_\alpha)_{\alpha\in\Gamma}),\sigma_\varphi^t(
(w_\alpha)_{\alpha\in\Gamma}))}\:,\]
which lead to the desired result.
\end{proof}
\noindent The following remark deal with the situation
$W(\varphi)\neq\varnothing$, and its connection to Li--Yorke and topological
chaoticity of $(X^\Gamma,\sigma_\varphi)$.
Let's recall that the system $(Z,f)$ is {\it
topologically chaotic} if it has positive topological entropy (for the
definition of topological entropy and more details see
\cite{walters}).
\begin{remark}\label{factA}
By \cite[Theorem 4.7]{set}, topological entropy of
$\sigma_\varphi:X^\Gamma\to X^\Gamma$ is equal to ${\mathfrak o}(\varphi)\log|X|$, where:
\begin{center}
${\mathfrak o}(\varphi) =\sup(\{0\}\cup\{n\in\mathbb{N}:$ there exist $\alpha_1
,\ldots,\alpha_n\in\Gamma$ such that $\{\varphi^m(\alpha_1)\}_{m\geq1},\ldots,\{\varphi^m(\alpha_n)\}_{m\geq1}$
are infinite and pairwise disjoint $\})$.
\end{center}
Hence $(X^\Gamma,\sigma_\varphi)$ is topological chaotic if and
only if ${\mathfrak o}(\varphi)>0 $, i.e.,
$\varphi:\Gamma\to\Gamma $ has at least one non-quasi--periodic
point.
So by \cite[Theorem 3.3]{li}
the system
$(X^\Gamma,\sigma_\varphi)$ is Li-Yorke chaotic (resp. has an
scrambled pair) if and only if the map $\varphi:\Gamma\to\Gamma$
has at least one non--quasi--periodic point which is equivalent to topological chaoticity of 
$(X^\Gamma,\sigma_\varphi)$ in its turn.
\end{remark}
\begin{lemma}\label{narjes10}
If $\varphi:\Gamma\to\Gamma$ does not have any periodic point,
then there exists $\mu>0$ such that:
\[\forall x\in X^\Gamma\:\:\exists y\in X^\Gamma\:\:((x,y)\in
S(X^\Gamma,\sigma_\varphi)\wedge{\displaystyle\limsup_{t\to\infty}D(\sigma_\varphi^t(
x),\sigma_\varphi^t(
y))}=\mu)\:.\]
\end{lemma}
\begin{proof}
Since $\varphi:\Gamma\to\Gamma$ does not have any periodic point,
it does not have any quasi--periodic point too, and all of points of
$\Gamma$ are non--quasi--periodic, hence by Remark~\ref{factA} there
exists
$((p_\alpha)_{\alpha\in\Gamma},(q_\alpha)_{\alpha\in\Gamma})\in
S(X^\Gamma,\sigma_\varphi)$, let:
\[\mu:={\displaystyle\limsup_{t\to\infty}D(\sigma_\varphi^t(
(p_\alpha)_{\alpha\in\Gamma}),\sigma_\varphi^t(
(q_\alpha)_{\alpha\in\Gamma}))}\:.\]
Consider
$(x_\alpha)_{\alpha\in\Gamma}\in X^\Gamma$ and choose
$(y_\alpha)_{\alpha\in\Gamma}\in X^\Gamma$ such that:
\[y_\alpha\left\{\begin{array}{lc} =x_\alpha & p_\alpha=q_\alpha\: , \\
\in \{p_\alpha,q_\alpha\}\setminus\{x_\alpha\} & p_\alpha\neq q_\alpha\:.
\end{array}\right.\]
Hence $x_\alpha=y_\alpha$ if and only if $p_\alpha=q_\alpha$, therefore
\[\delta(x_\alpha,y_\alpha)=\delta(p_\alpha,q_\alpha)\SP(
\forall\alpha\in\Gamma)\:.\]
For all $t\geq1$ we have:
\begin{align*}
D(\sigma_\varphi^t((x_\alpha)_{\alpha\in\Gamma}),\sigma_\varphi^t(
(y_\alpha)_{\alpha\in\Gamma})) & = 
{\displaystyle\sum_{i\geq1}\frac{\delta(x_{\varphi^t(\beta_i)},
y_{\varphi^t(\beta_i)})}{2^i}}  =  {\displaystyle\sum_{i\geq1}\frac{\delta(p_{\varphi^t(\beta_i)},
q_{\varphi^t(\beta_i)})}{2^i}} \\
& =  D(\sigma_\varphi^t((p_\alpha)_{\alpha\in\Gamma}),\sigma_\varphi^t(
(q_\alpha)_{\alpha\in\Gamma}))
\end{align*}
which leads to
{\small \[{\displaystyle\liminf_{t\to\infty}
D(\sigma_\varphi^t((x_\alpha)_{\alpha\in\Gamma}),\sigma_\varphi^t(
(y_\alpha)_{\alpha\in\Gamma})) }={\displaystyle\liminf_{t\to\infty}
D(\sigma_\varphi^t((p_\alpha)_{\alpha\in\Gamma}),\sigma_\varphi^t(
(q_\alpha)_{\alpha\in\Gamma})) }=0\:,\]
 \[{\displaystyle\limsup_{t\to\infty}
D(\sigma_\varphi^t((x_\alpha)_{\alpha\in\Gamma}),\sigma_\varphi^t(
(y_\alpha)_{\alpha\in\Gamma})) }={\displaystyle\limsup_{t\to\infty}
D(\sigma_\varphi^t((p_\alpha)_{\alpha\in\Gamma}),\sigma_\varphi^t(
(q_\alpha)_{\alpha\in\Gamma})) }=\mu>0\:,\]}
and $((x_\alpha)_{\alpha\in\Gamma},(y_\alpha)_{\alpha\in\Gamma})\in
S(X^\Gamma,\sigma_\varphi)$.
\end{proof}
\noindent Now we have the
following lemma.
\begin{lemma}\label{javad10}
If $(X^\Gamma,\sigma_\varphi)$ is densely chaotic, then $W(\varphi)=\Gamma$ (i.e.,
$\varphi:\Gamma\to\Gamma$ does not have any periodic point).
\end{lemma}
\begin{proof}
Consider $\beta\in Per(\varphi)$ and $n\geq1$ with
$\varphi^n(\beta)=\beta$, we prove
$(X^\Gamma,\sigma_\varphi)$ is not densely chaotic.
We may also suppose
$\beta_1=\beta,\beta_2=\varphi(\beta),\ldots,\beta_n=\varphi^{n-1}(\beta)$.
Choose distinct $p,q\in X$ and let:
\[U_\alpha=\left\{\begin{array}{lc}
\{p\}&\alpha=\beta_1,\ldots,\beta_n\:, \\ X &
{\rm otherwise\:,}\end{array}\right.
\SP\SP V_\alpha=\left\{\begin{array}{lc}
\{q\}&\alpha=\beta_1,\ldots,\beta_n\:, \\ X &
{\rm otherwise\:,}\end{array}\right.\]
also let:
\[U={\displaystyle\prod_{\alpha\in\Gamma}U_\alpha}\SP,\SP
V={\displaystyle\prod_{\alpha\in\Gamma}V_\alpha}\:,\]
then $U\times V$ is an opene subset of $X^\Gamma\times X^\Gamma$.
Consider
$(x,y)=((x_\alpha)_{\alpha\in\Gamma},(y_\alpha)_{\alpha\in\Gamma})\in U\times V$
we have $x_{\beta_1}=\cdots=x_{\beta_n}=p$ and
$y_{\beta_1}=\cdots=y_{\beta_n}=q$.
For all $k\geq1$ we have
$\varphi^k(\beta)\in\{\beta_1,\ldots,\beta_n\}$ which leads to
$x_{\varphi^k(\beta)}=p$ and $y_{\varphi^k(\beta)}=q$, thus:
\begin{align*}
D(\sigma_\varphi^k((x_\alpha)_{\alpha\in\Gamma}
,(y_\alpha)_{\alpha\in\Gamma})) & = D
((x_{\varphi^k(\alpha)})_{\alpha\in\Gamma},
(y_{\varphi^k(\alpha)})_{\alpha\in\Gamma}) \\
& \geq  \frac12\delta(x_{\varphi^k(\beta)},y_{\varphi^k(\beta)})
=\frac12\delta(p,q)=\frac12\:.
\end{align*}
Hence
\[{\displaystyle\liminf_{k\to\infty}
D(\sigma_\varphi^k((x_\alpha)_{\alpha\in\Gamma},
(y_\alpha)_{\alpha\in\Gamma}))}\geq\frac12\]
and $(x,y)\notin S(X^\Gamma,\sigma_\varphi)$.
Using $(U\times V)\cap S(X^\Gamma,\sigma_\varphi)=\varnothing$
we have the desired result.
\end{proof}
\noindent Now we are
ready to characterize densely chaotic generalized shifts.
\begin{theorem}[Densely chaotic generalized shifts]\label{factB}
For finite discrete $X$ with at least two elements, nonempty countable set $\Gamma$
and $\varphi:\Gamma\to\Gamma$, the following statements are equivalent:
\begin{itemize}
\item[1.] for some $\varepsilon>0$,
$(X^\Gamma,\sigma_\varphi)$
is densely $\varepsilon-$chaotic;
\item[2.] the system $(X^\Gamma,\sigma_\varphi)$
is densely chaotic;
\item[3.] the map $\varphi:\Gamma\to\Gamma$ does not have any periodic point
(i.e., $W(\varphi)=\Gamma$).
\end{itemize}
\end{theorem}
\begin{proof}
Clearly (1) implies (2). By Lemma~\ref{javad10}, 
(2) implies (3). 
\\
(3 $\Rightarrow$  1): Suppose  $W(\varphi)=\Gamma$
by Lemma~\ref{narjes10} there exists $\mu>0$
such that for all $x=(x_\alpha)_{\alpha\in\Gamma}\in
X^\Gamma$, there exists
$y=(y_\alpha)_{\alpha\in\Gamma}\in X^\Gamma$ with
$((x_\alpha)_{\alpha\in\Gamma},(y_\alpha)_{\alpha\in\Gamma})\in
S(X^\Gamma,\sigma_\varphi)$ and
\[{\displaystyle\limsup_{t\to\infty}
D(\sigma_\varphi^t((x_\alpha)_{\alpha\in\Gamma}),\sigma_\varphi^t(
(y_\alpha)_{\alpha\in\Gamma})) }=\mu\:.\]
Suppose $W$ is an opene subset of $X^\Gamma\times X^\Gamma$, there exist opene subsets $U,V$ of $X^\Gamma$ with $U\times V\subseteq W$. Choose $u=(u_\alpha)_{\alpha\in\Gamma}\in U,
v=(v_\alpha)_{\alpha\in\Gamma}\in V$ and $\psi_1,\ldots,\psi_n\in\Gamma$ such that 
${\displaystyle\prod_{\alpha\in\Gamma}U_\alpha}\subseteq U,
{\displaystyle\prod_{\alpha\in\Gamma}V_\alpha}\subseteq V$ with
\[U_\alpha=\left\{\begin{array}{lc} \{u_\alpha\} & \alpha=\psi_1,\ldots,\psi_n\:, \\ X & {\rm otherwise\:,}
\end{array}\right. \SP and \SP
V_\alpha=\left\{\begin{array}{lc} \{v_\alpha\} & \alpha=\psi_1,\ldots,\psi_n\:, \\ X & {\rm otherwise\:.}
\end{array}\right. \]
Then for:
\[z_\alpha=\left\{\begin{array}{lc} u_\alpha & \alpha=\psi_1,\ldots,\psi_n\:, \\ x_\alpha & {\rm otherwise\:,}
\end{array}\right. \SP and \SP
w_\alpha=\left\{\begin{array}{lc} v_\alpha & \alpha=\psi_1,\ldots,\psi_n\:, \\ y_\alpha & {\rm otherwise\:.}
\end{array}\right. \]
by Lemma~\ref{taha50} we have $(z,w):=((z_\alpha)_{\alpha\in\Gamma},(w_\alpha)_{\alpha\in\Gamma})\in S(X^\Gamma,\sigma_\varphi)$ with 
\[{\displaystyle\limsup_{t\to\infty}D(\sigma_\varphi^t(z),\sigma_\varphi^t(w)) }={\displaystyle\limsup_{t\to\infty}D(\sigma_\varphi^t(x),\sigma_\varphi^t(y)) }=\mu\:.\]
Thus $(z,w)\in (U\times V)\cap S_{\frac\mu2}(X^\Gamma,\sigma_\varphi)$
and $W\cap S_{\frac\mu2}(X^\Gamma,\sigma_\varphi)\neq\varnothing$ for each opene subset $W$ of 
$X^\Gamma\times X^\Gamma$. Therefore $(X^\Gamma,\sigma_\varphi)$ is $\frac\mu2-$densely chaotic.
\end{proof}
\section{Sensitivity in generalized shift dynamical systems}
\noindent We recall that the dynamical system $(Z,f)$ with compact
metric phase space $(Z,d)$ is \cite{sharma}:
\begin{itemize}
\item {\it sensitive} if there exists $\varepsilon>0$ such that for
all $x\in Z$ and open neighbourhood $V$ of $x$ there exist $n\geq0$
and $y\in V$ with $d(f^n(x),f^n(y))>\varepsilon$;
\item {\it strongly sensitive}, if there exists $\varepsilon>0$ such that for
all $x\in Z$ and open neighbourhood $V$ of $x$ there exist $n_0\geq0$
and $y\in V$ with $d(f^n(x),f^n(y))>\varepsilon$ for all $n\geq n_0$.
\end{itemize}
As it has been mentioned in
\cite[Theorem 3]{kolyada},
$(Z,f)$ is sensitive if and only if there exists $\varepsilon>0$
such that $\{(x,y)\in Z\times Z:{\displaystyle\limsup_{n\to\infty}d(f^n(x),f^n(y))}>
\varepsilon\}$ is dense in $Z\times Z$. Using a similar method
described in Note~\ref{taha}
being sensitive (resp. strongly sensitive)
does not depend on compatible metric on $Z$.
Now we are ready to prove that the system $(X^\Gamma,\sigma_\varphi)$
is sensitive (resp. strongly sensitive)
if and only if it is Li-Yorke chaotic, i.e.
$W(\varphi)\neq\varnothing$.
\begin{theorem}\label{taha80}
If $W(\varphi)\neq\varnothing$, then $(X^\Gamma,\sigma_\varphi)$
is strongly sensitive.
\end{theorem}
\begin{proof}
Suppose $W(\varphi)\neq\varnothing$ and choose
$\theta\in W(\varphi)$, we may suppose $\theta=\beta_k$ (see Convention~\ref{taha10}).
If $U$ is an open neighbourhood of $x=(x_\alpha)_{\alpha\in\Gamma}\in X^\Gamma$, there exists finite subset $F$ of $\Gamma$
such that for $\{(y_\alpha)_{\alpha\in\Gamma}\in X^\Gamma:\forall\alpha\in F\: y_\alpha=x_\alpha\}\subseteq U$.
Since $\{\varphi^n(\theta)\}_{n\geq1}$ is a one to one
sequence, there exists $N\geq1$ such that for all $n\geq N$ we have
$\varphi^n(\theta)\notin F$, i.e.
$\{\varphi^n(\theta):n\geq N\}\subseteq \Gamma\setminus F$. So
\[\{(y_\alpha)_{\alpha\in\Gamma}\in X^\Gamma:\forall\alpha\neq\varphi^N(\theta),
\varphi^{N+1}(\theta),\varphi^{N+2}(\theta),\ldots,\:y_\alpha=x_\alpha\}\subseteq U\: .\]
For all $m\geq N$ choose $p_m\in X\setminus\{x_{\varphi^m(\theta)}\}$, also let
\[z_\alpha=\left\{ \begin{array}{lc} p_m & \alpha=\varphi^m(\theta), m\geq N\:, \\
x_\alpha & {\rm otherwise\: .}\end{array}\right.\]
Then $(z_\alpha)_{\alpha\in\Gamma}\in U$ and
for $m\geq N$,
$(u_\alpha)_{\alpha\in\Gamma}:=\sigma_\varphi^m((z_\alpha)_{\alpha\in\Gamma})$,
$(v_\alpha)_{\alpha\in\Gamma}:=\sigma_\varphi^m((x_\alpha)_{\alpha\in\Gamma})$
we have
\begin{align*}
D(\sigma_\varphi^m((x_\alpha)_{\alpha\in\Gamma}),
\sigma_\varphi^m((z_\alpha)_{\alpha\in\Gamma})) &= D((v_\alpha)_{\alpha\in\Gamma},
(u_\alpha)_{\alpha\in\Gamma}) \\
& \geq  \dfrac{\delta(v_\theta,u_\theta)}{2^k} =  \dfrac{\delta(x_{\varphi^m(\theta)},z_{\varphi^m(\theta)})}{2^k} \\
& =  \dfrac{\delta(x_{\varphi^m(\theta)},p_m)}{2^k}=\dfrac1{2^k}\:.
\end{align*}
Hence $(X^\Gamma,\sigma_\varphi)$
is strongly sensitive.
\end{proof}
\begin{theorem}\label{taha90}
If $W(\varphi)=\varnothing$, then $(X^\Gamma,\sigma_\varphi)$
is not sensitive.
\end{theorem}
\begin{proof}
Suppose $W(\varphi)=\varnothing$ and consider arbitrary
$\varepsilon>0$, then there exists $N\geq1$ such that
$\dfrac1{2^N}<\varepsilon$. Since $W(\varphi)=\varnothing$,
for all $\alpha\in\Gamma$ the set $\{\varphi^n(\alpha):n\geq0\}$
is finite. Thus
\[\Lambda:=\{\varphi^n(\beta_i):i\in\{1,\ldots,N\},n\geq0\}\]
is finite too, for $x=(x_\alpha)_{\alpha\in\Gamma}\in X^\Gamma$,
$U=\{(y_\alpha)_{\alpha\in\Gamma}\in X^\Gamma:\forall\alpha\in\Lambda\:(
y_\alpha=x_\alpha)\}$ is an open neighbourhood of
$(x_\alpha)_{\alpha\in\Gamma}\in X^\Gamma$.
For $n\geq0$ and $(y_\alpha)_{\alpha\in\Gamma}\in U$
let $(v_\alpha)_{\alpha\in\Gamma}:=\sigma_\varphi^n((x_\alpha)_{\alpha\in\Gamma})$
and $(w_\alpha)_{\alpha\in\Gamma}:=\sigma_\varphi^n((y_\alpha)_{\alpha\in\Gamma})$,
then for all $\alpha\in\Lambda$ we have $\varphi^n(\alpha)\in\Lambda$
and $v_\alpha=x_{\varphi^n(\alpha)}=y_{\varphi^n(\alpha)}=w_\alpha$, thus
\[D(\sigma_\varphi^n((x_\alpha)_{\alpha\in\Gamma}),
\sigma_\varphi^n((y_\alpha)_{\alpha\in\Gamma}))
={\displaystyle\sum_{i\geq1,\beta_i\notin\Lambda}\dfrac{\delta(v_{\beta_i},w_{\beta_i})}{2^i}}
\leq{\displaystyle\sum_{i>N}\dfrac1{2^i}}=\dfrac1{2^N}<\varepsilon\:.\]
So for all $\varepsilon>0$ and $x\in X^\Gamma$ there exists open neighbourhood $U$
of $x$ such that $D(\sigma_\varphi^n(x),\sigma_\varphi^n(y))<\varepsilon$
for all $y\in U$ and $n\geq0$, which leads to the desired result.
\end{proof}
\noindent Using Lemmas~\ref{narjes10} and \ref{taha50}, we have
the following theorem.
\begin{theorem}\label{javad20}
If $\varphi:\Gamma\to\Gamma$ does not have any periodic point,
then $(X^\Gamma,\sigma_\varphi)$
is Li-Yorke sensitive.
\end{theorem}
\begin{proof}
Suppose $\varphi:\Gamma\to\Gamma$ does not have
any periodic point then by Lemma~\ref{narjes10} there exists $\mu>0$
such that for all $a\in
X^\Gamma$, there exists
$b_a\in X^\Gamma$ with
\[(a,b_a)\in
S(X^\Gamma,\sigma_\varphi)\wedge
{\displaystyle\limsup_{t\to\infty}
D(\sigma_\varphi^t(a),\sigma_\varphi^t(
b_a)) }=\mu\:.\]
Consider $x=(x_\alpha)_{\alpha\in\Gamma}\in X^\Gamma$, and $b_x=:y=(y_\alpha)_{\alpha\in\Gamma}$.
For all $n\geq1$ define $y_n=(y_\alpha^n)_{\alpha\in\Gamma}$ with:
\[y^n_\alpha=\left\{\begin{array}{lc}
x_\alpha & \alpha\in\{\beta_1,\ldots,\beta_n\}\:, \\
y_\alpha & {\rm otherwise}\:. \end{array}\right.\]
Using Lemma~\ref{taha50}  for all $n\geq 1$ we have
\[{\displaystyle\limsup_{t\to\infty}
D(\sigma_\varphi^t(x),\sigma_\varphi^t(
y_n)) }=
{\displaystyle\limsup_{t\to\infty}
D(\sigma_\varphi^t(x),\sigma_\varphi^t(
b_x) })=\mu\]
and
\[{\displaystyle\liminf_{t\to\infty}
D(\sigma_\varphi^t(x),\sigma_\varphi^t(
y_n)) }=
{\displaystyle\liminf_{t\to\infty}
D(\sigma_\varphi^t(x),\sigma_\varphi^t(
b_x) })=0\:,\]
in particular
$(x,y_n)\in S_{\frac\mu2}(X^\Gamma,\sigma_\varphi)$. Moreover by 
${\displaystyle\lim_{n\to\infty}y_n}=
x$ for all open neighbourhood $U$ of
$x$ there exists $n\in\mathbb N$ with
$y_n\in U$ which completes the proof.
\end{proof}
\begin{theorem}\label{taha95}
The generalized shift dynamical system $(X^\Gamma,\sigma_\varphi)$ is Li--Yorke sensitive
if and only if $\varphi:\Gamma\to\Gamma$ has at least one non--quasi--periodic point.
\end{theorem}
\begin{proof}
 If $(X^\Gamma,\sigma_\varphi)$ is Li--Yorke sensitive, then $S(X^\Gamma,\sigma_\varphi)\neq\varnothing$ and
 by Remark~\ref{factA} $\varphi$ has a non--quasi--periodic point.
 \\
 On the other hand if $\varphi$ has a non--quasi--periodic point $\theta\in\Gamma$, then
 for $\Lambda=\bigcup\{\varphi^n(\theta):n\in{\mathbb Z}\}$, $\varphi\restriction_\Lambda:\Lambda\to\Lambda$
 does not have any periodic point and by Theorem~\ref{javad20} $(X^\Lambda,\sigma_{\varphi\restriction_\Lambda})$
 is Li--Yorke sensitive. We may suppose $\Lambda=\{\beta_{s_1},\beta_{s_2},\ldots\}$ with $s_1<s_2<\cdots$, consider $p\in X$
 and equip $X^\Lambda$ with metric $D_\Lambda(x,y)=D(x^*,y^*)$, where for $x=(x_\alpha)_{\alpha\in\Lambda}\in X^\Lambda$ we have $x^*_\alpha=x_\alpha$ for $\alpha\in\Lambda$, $x^*_\alpha=p$ for $\alpha\notin\Lambda$,
 and $x^*=(x_\alpha^*)_{\alpha\in\Gamma}$.
Since $(X^\Lambda,\sigma_{\varphi\restriction_\Lambda})$
 is Li--Yorke sensitive there exists $\kappa>0$ such that for all $x\in X^\Lambda$ and $\varepsilon>0$ there exists
 $y\in X^\Lambda$ with $D_\Lambda(x,y)<\varepsilon$ and $(x,y)\in S_\kappa(X^\Lambda,\sigma_{\varphi\restriction_\Lambda})$.
\\
Consider $z=(z_\alpha)_{\alpha\in\Gamma}\in X^\Gamma$ and open neighbourhood $V_0$ of $z$. There 
exists $r>0$ with $V:=\{a\in X^\Gamma:D(a,z)<r\}\subseteq V_0$,
Natural projection map
$p_\Lambda:\mathop{X^\Gamma\to X^\Lambda}\limits_{(x_\alpha)_{\alpha\in\Gamma}\mapsto(x_\alpha)_{\alpha\in\Lambda}}$
 is open and continuous hence there exists $\varepsilon>0$ such that
 \[\{a\in X^\Lambda:D_\Lambda(p_\Lambda(z),a)<\varepsilon\}\subseteq p_\Lambda(V)\:,\]
 so there exists 
 $w\in \{a\in X^\Lambda:D_\Lambda(p_\Lambda(z),a)<\varepsilon\}$ and
 $y\in V$ with $p_\Lambda(y)=w$ and 
 $(p_\Lambda(z),w)\in S_\kappa(X^\Lambda,\sigma_{\varphi\restriction_\Lambda})$,
 i.e. $(p_\Lambda(z),p_\Lambda(y))\in S_\kappa(X^\Lambda,\sigma_{\varphi\restriction_\Lambda})$.
 For $h=(h_\alpha)_{\alpha\in\Lambda}$ let:
 \[\overline{h}_\alpha:=\left\{\begin{array}{lc}h_\alpha & \alpha\in\Lambda\:, \\ z_\alpha & {\rm otherwise\:,}
 \end{array} \right. \]
 and $\overline{h}=(\overline{h}_\alpha)_{\alpha\in\Gamma}$.
 For $t\geq0$ we have $D(\sigma_\varphi^t(z),\sigma_\varphi^t(\overline{p_\Lambda(y)})=D_\Lambda(\sigma^t_{\varphi
 \restriction_\Lambda}(p_\Lambda(z)),\sigma^t_{\varphi
 \restriction_\Lambda}(p_\Lambda(y))$, therefore $(z,\overline{p_\Lambda(y)})\in S_\kappa(X^\Gamma,\sigma_\varphi)$.
 By $D(z,\overline{p_\Lambda(y)})\leq D(z,y)<r$ we have $\overline{p_\Lambda(y)}\in V\subseteq V_0$ and obtain the desired result.
\end{proof}
\noindent Using Theorems \ref{taha80}, \ref{taha90} and \ref{taha95}
we have the following theorem:
\begin{theorem}[sensitive generalized shifts]\label{110}
The following statements are equivalent:
\begin{itemize}
\item[1.] $(X^\Gamma,\sigma_\varphi)$ is strongly sensitive;
\item[2.] $(X^\Gamma,\sigma_\varphi)$ is sensitive;
\item[3.] $(X^\Gamma,\sigma_\varphi)$ is Li--Yorke sensitive;
\item[4.] $\varphi:\Gamma\to\Gamma$ has at least one non--quasi--periodic point.
\end{itemize}
\end{theorem}
\begin{proof}
It's clear that (1) implies (2). By Theorem~\ref{taha90}, (2) implies (4). By Theorem~\ref{taha80}, (4) implies (1). By Theorem~\ref{taha95}, (3) and (4) are equivalent.
\end{proof}
\subsection{Two diagrams}
\noindent By \cite[Theorem 2.13]{dev} the system $(X^\Gamma,\sigma_\varphi)$
is Devaney chaotic (resp. topological transitive)
if and only if $\varphi:\Gamma\to\Gamma$ is one to one without periodic points, now we are
ready to summarize Sections 3 and 4 in the following diagram, which
completes the mentioned diagram in Introduction:
\[\scalebox{0.8}{\xymatrix{\varphi{\rm \: is \: one \: to \: one \: and\:}Per(\varphi)=\varnothing\ar@{=>}[r]
\ar@{<=>}[d] &
W(\varphi)=\Gamma\ar@{=>}[r]\ar@{<=>}[d] & W(\varphi)\neq\varnothing\ar@{<=>}[d] \\
(X^\Gamma,\sigma_\varphi)\:{\rm is \: Devaney \: chaotic} &
(X^\Gamma,\sigma_\varphi)\: {\rm is \: densely \: chaotic} &
(X^\Gamma,\sigma_\varphi)\: {\rm is \: Li-Yorke \: chaotic}}}\]
Let:

$\mathcal{C}:=\{(X^\Gamma,\sigma_\eta):\eta\in\Gamma^\Gamma\}$, 

$\mathcal{C}_{Devaney}:=\{(X^\Gamma,\sigma_\eta)\in\mathcal{C}:(X^\Gamma,\sigma_\eta)$
is Devaney chaotic$\}$, 

$\mathcal{C}_{Densely}:=\{(X^\Gamma,\sigma_\eta)\in\mathcal{C}:(X^\Gamma,\sigma_\eta)$
is Densely chaotic$\}$, 

$\mathcal{C}_{LY}:=\{(X^\Gamma,\sigma_\eta)\in\mathcal{C}:(X^\Gamma,\sigma_\eta)$
is Li--Yorke chaotic$\}$.
\\
For infinite $\Gamma$, suppose $\beta_n$s are distinct, then we have the following diagram:
\begin{center}
{\small
\begin{tabular}{|c|}
\hline 
$\mathcal{C}$ \\
	\begin{tabular}{|c|} \hline
	$\mathcal{C}_{LY}$ \\
		\begin{tabular}{|c|} \hline
		$\mathcal{C}_{Densely}$ \\
			\begin{tabular}{|c|} \hline
			$\mathcal{C}_{Devaney}$ \\
			E1 \\ \hline
			\end{tabular}\\
		E2 \\ \hline
		\end{tabular}\\
	E3 \\ \hline
	\end{tabular} \\
E4 \\ \hline
\end{tabular}}
\end{center}
Where ``Ei'' denotes Example $(X^\Gamma,\sigma_{\varphi_i})$ for $\varphi_i:\Gamma\to\Gamma$ 
with:
\begin{itemize}
\item $\varphi_1(\beta_n)=\beta_{2n}$ for $n\geq1$, 
\item $\varphi_2(\beta_1)=\varphi_2(\beta_2)=\beta_3$, and $\varphi_2(\beta_n)=\beta_{2n}$ for $n\geq3$,
\item $\varphi_3(\beta_1)=\varphi_3(\beta_2)=\beta_2$, and $\varphi_3(\beta_n)=\beta_{2n}$ for $n\geq3$,
\item $\varphi_4(\beta_1)=\beta_1$, and $\varphi_4(\beta_n)=\beta_{n-1}$ for $n\geq2$.
\end{itemize}
\subsection{On more types of sensitivity} 
\noindent In this subsection consider dynamical system $(Z,f)$ with compact 
metric phase space $(Z,d)$. 
\begin{note}\label{setareh1}
According
to \cite{sharma}
the dynamical system $(Z,f)$  is
{asymptotic sensitive} if there exists $\varepsilon>0$ such that for
all $x\in Z$ and open neighbourhood $V$ of $x$ there exists
$y\in V$ with 
\linebreak
${\displaystyle\limsup_{n\to\infty} d(f^n(x),f^n(y))}>\varepsilon$.
Let's verify the following diagram:
\begin{center}
strongly sensitive $\Rightarrow$ asymptotic sensitive $\Rightarrow$ sensitive.
\end{center}
Note that if $(Z,f)$ is strongly sensitive, then there exists $\kappa>0$ such that for
all $x\in Z$ and open neighbourhood $V$ of $x$ there exist $n_0\geq0$
and $y\in V$ with $d(f^n(x),f^n(y))>\kappa$ for all $n\geq n_0$, hence
\[{\displaystyle \limsup_{n\to\infty} d(f^n(x),f^n(y))}=
{\displaystyle \limsup_{\mathop{n\to\infty}\limits_{n\geq n_0}} d(f^n(x),f^n(y))}\geq\kappa>\frac\kappa2=:\mu\:,\]
thus $(Z,f)$ is asymptotic sensitive.
\\
Moreover if $(Z,f)$ is asymptotic sensitive, then 
there exists $\mu>0$ such that for
all $x\in Z$ and open neighbourhood $V$ of $x$ there exists
$y\in V$ with 
\linebreak
${\displaystyle\limsup_{n\to\infty} d(f^n(x),f^n(y))}>\mu$, thus
$\{n\geq1:d(f^n(x),f^n(y))>\mu\}$ is infinite and in particular it is nonempty. Hence $(Z,f)$ is  sensitive.
\end{note}
\noindent For $\varepsilon>0$ and open subset $V$ of $Z$ let
$N(V,\varepsilon):=\{n\geq0:\exists x,y\in V\:(d(f^n(x),f^n(y))>\varepsilon)\}$
and we call $A\subseteq\mathbb{N}\cup\{0\}$ {\it syndetic}
if there exists $N\geq1$ with $\{i,i+1,\ldots,i+N\}\cap A\neq\varnothing$
for all $i\geq 1$. 
\begin{note}\label{setareh}
$(Z,f)$ is sensitive if and only if there exists $\kappa>0$ such that for all opene subset $V$ of $Z$,
$N(V,\kappa)\neq\varnothing$.
\\
First suppose $(Z,f)$ is sensitive, then there exists $\varepsilon>0$ such that for all $x\in X$ and open neighbourhood
$V$ of $x$ there exists $y\in V$ and $n\geq0$ with $d(f^n(x),f^n(y))>\varepsilon$. For opene subset $W$
of $Z$ choose $z_1\in W$, then there exists $z_2\in W$ and $m\geq0$ with $d(f^m(z_1),f^m(z_2))>\varepsilon$,
thus $m\in N(W,\varepsilon)$ and $N(W,\varepsilon)\neq\varnothing $ for all opene subset $W$ of $Z$.
\\
Now suppose there exists $\mu>0$ such that $N(V,\kappa)\neq\varnothing$ for all opene subset $V$ of $Z$.
For all $x\in Z$ and open neighbouhhood $V$ of $x$ we have $N(V,\kappa)\neq\varnothing$ thus there exist
$y,z\in V$ and $n\geq0$ with $d(f^n(z),f^n(y))>\kappa$, therefore $d(f^n(x),f^n(y))>\frac\kappa2$ or
$d(f^n(x),f^n(z))>\frac\kappa2$. Thus $(Z,f)$ is sensitive.
\end{note}
\begin{note}\label{setareh2}
According to \cite{chin} we call
$(Z,f)$ {\it syndetically sensitive} (resp. {\it cofinitely sensitive})
if there exists $\varepsilon>0$ such that for all  opene subset
$V$ of $Z$, $N(V,\varepsilon)$ is syndetic (resp. cofinite). Let's verify the following diagram:
\begin{center}
strongly sensitive $\Rightarrow$ cofinitely sensitive
$\Rightarrow$ syndetically sensitive $\Rightarrow$ sensitive.
\end{center}
Note that if $(Z,f)$ is strongly sensitive, then there exists $\kappa>0$ such that for
all $x\in Z$ and open neighbourhood $V$ of $x$ there exist $n_0\geq0$
and $y\in V$ with $d(f^n(x),f^n(y))>\kappa$ for all $n\geq n_0$. 
Thus for opene $W$ of $Z$ choose $z\in W$, there exist $m\geq0$ and $y\in W$
with $d(f^n(z),f^n(y))>\kappa$ for all $n\geq m$, thus 
$\{m,m+1,\ldots\}\subseteq \{n\geq0:\exists p,q\in W\:(d(f^n(p),f^n(q))>\kappa)\}=N(W,\kappa)$
and $N(W,\kappa)$ is cofinite, Hence $(Z,f)$ is cofinitely sensitive.
\\
Since any cofinite subset of $\mathbb{N}\cup\{0\}$ is syndetic, if $(Z,f)$ is cofinitely sensitive, then it is 
syndetically sensitive.
Since any syndetic subset of $\mathbb{N}\cup\{0\}$ is nonempty, if $(Z,f)$ is syndetically sensitive, then it is sensitive
(use Note~\ref{setareh}).
\end{note}
\begin{note}\label{setareh3}
Regarding \cite{chin} we call $(Z,f)$
{\it multi--sensitive} if
there exists $\varepsilon>0$ such that for all $k\geq1$ and
 opene subsets
$V_1,\ldots,V_k$ of $Z$, we have ${\displaystyle\bigcap_{1\leq n\leq k}
N(V_n,\varepsilon)}\neq\varnothing$. Let's verify the following diagram:
\begin{center}
strongly sensitive $\Rightarrow$ multi-sensitive $\Rightarrow$ sensitive.
\end{center}
Note that if $(Z,f)$ is strongly sensitive, then there exists $\kappa>0$ such that for
all $x\in Z$ and open neighbourhood $V$ of $x$ there exist $n_0\geq0$
and $y\in V$ with $d(f^n(x),f^n(y))>\kappa$ for all $n\geq n_0$. Suppose
$V_1\ldots,V_k$ are opene subsets of $Z$, for each $i$ choose $x_i\in V_i$, then
there exists $n_i\geq0$
and $y_i\in V_i$ with $d(f^n(x_i),f^n(y_i))>\kappa$ for all $n\geq n_i$. Thus
for $m=\max(n_1,\cdots,n_k)$ we have $\{m,m+1,\ldots\}\subseteq 
{\displaystyle\bigcap_{1\leq i\leq k}
N(V_i,\kappa)}$, in particular ${\displaystyle\bigcap_{1\leq i\leq k}
N(V_i,\kappa)}\neq\varnothing$ and $(Z,f)$is multi--sensitive.
\\
By Note~\ref{setareh}, if $(Z,f)$ is multi--sensitive, then it is sensitive.
\end{note}
\begin{lemma}\label{lemsetareh}
If $(Z,f)$ is cofinitely sensitive, then there exists 
$\varepsilon>0$ such that 
\[{\displaystyle\lim_{n\to\infty}\frac{|N(V,\varepsilon)\cap\{0,\ldots,n\}|}{n+1}}=1\]
for all opene $V$ subset of $Z$.
\end{lemma}
\begin{proof}
Since $(Z,f)$ is cofinitely sensitive there exists $\varepsilon>0$ such that for all  opene subset
$V$ of $Z$, $N(V,\varepsilon)$ is cofinite. Hence for opene subset $V$ of $Z$ there exists
$m\geq1$ such that $\{m,m+1,\ldots\}\subseteq N(V,\varepsilon)$, thus for all $n\geq m$ we have 
\[\frac{n-m+1}{n+1}=\frac{|\{m,m+1,\ldots n\}\cap\{0,\ldots,n\}|}{n+1}\leq \frac{|N(V,\varepsilon)\cap\{0,\ldots,n\}|}{n+1}\leq1\]
which leads to 
${\displaystyle\lim_{n\to\infty}\frac{|N(V,\varepsilon)\cap\{0,\ldots,n\}|}{n+1}}=1$ and completes the proof.
\end{proof}
\begin{note}\label{setareh4}
We call $(Z,f)$ {\it ergodically sensitive} \cite{chin}
if there exists $\varepsilon>0$ such that for all  opene subset
$V$ of $Z$,
${\displaystyle\limsup_{n\to\infty}\frac{|N(V,\varepsilon)\cap\{0,\ldots,n\}|}{n+1}}>0$. 
By Notes~\ref{setareh2},~\ref{setareh} and Lemma~\ref{lemsetareh} it is easy to see that we have
the following diagram:
\begin{center}
strongly sensitive $\Rightarrow$ cofinitely sensitive $\Rightarrow$ ergodically sensitive $\Rightarrow$ sensitive.
\end{center}
\end{note}
\begin{theorem}
For finite discrete $X$ with at least two elements, nonempty countable set $\Gamma$ and
$\varphi:\Gamma\to\Gamma$ the following statements are equivalent:
\begin{itemize}
\item[1.] the system $(X^\Gamma,\sigma_\varphi)$ is
Li-Yorke chaotic (i.e. $(X^\Gamma,\sigma_\varphi)$ has an scrambled pair by \cite[Theorem 3.3]{li});
\item[2.] the system $(X^\Gamma,\sigma_\varphi)$ is topological chaotic;
\item[3.] the system $(X^\Gamma,\sigma_\varphi)$ is spatiotemporally chaotic;
\item[4.] the system $(X^\Gamma,\sigma_\varphi)$ is sensitive;
\item[5.] the system $(X^\Gamma,\sigma_\varphi)$ is Li--Yorke sensitive;
\item[6.] the system $(X^\Gamma,\sigma_\varphi)$ is strongly sensitive;
\item[7.] the system $(X^\Gamma,\sigma_\varphi)$ is asymptotic sensitive;
\item[8.] the system $(X^\Gamma,\sigma_\varphi)$ is syndetically sensitive;
\item[9.] the system $(X^\Gamma,\sigma_\varphi)$ is cofinitely sensitive;
\item[10.] the system $(X^\Gamma,\sigma_\varphi)$ is multi-sensitive;
\item[11] the system $(X^\Gamma,\sigma_\varphi)$ is ergodically sensitive;
\item[12.] the map $\varphi:\Gamma\to\Gamma$ has at least non--quasi--periodic point.
\end{itemize}
\end{theorem}
\begin{proof}
(1), (2), and (12) are equivalent by Remark~\ref{factA}.
\\
(4), (5), (6), and (12) are equivalent by Theorem~\ref{110}.
\\
(4), and (7) are equivalent by Theorem~\ref{110} and Note~\ref{setareh1}.
\\
(4), (8), and (9) are equivalent by Theorem~\ref{110} and Note~\ref{setareh2}.
\\
(4), and (10) are equivalent by Theorem~\ref{110} and Note~\ref{setareh3}.
\\
(4), and (11) are equivalent by Theorem~\ref{110} and Note~\ref{setareh4}.
\\
So (1) and  (5) are equivalent. 
In order to complete the proof, it's enough to show (5) imply (3), and (3) imply (1).
\\
(5$\Rightarrow$3): Suppose  $(X^\Gamma,\sigma_\varphi)$ is Li--Yorke sensitive. Then 
 there exists $\kappa>0$ such that
for every $x\in X^\Gamma$ and open neighbourhood $U$ of $x$ there exists
$y\in U$ with $(x,y)\in S_\kappa(X^\Gamma,\sigma_\varphi)$, since 
$S_\kappa(X^\Gamma,\sigma_\varphi)\subseteq S(X^\Gamma,\sigma_\varphi)$, we have
$(x,y)\in S(X^\Gamma,\sigma_\varphi)$ and $x,y$ are scrambled, hence 
$(X^\Gamma,\sigma_\varphi)$ is spatiotemporally chaotic.
\\
(3$\Rightarrow$1): Suppose  $(X^\Gamma,\sigma_\varphi)$ is spatiotemporally chaotic. Then
for every $x\in X^\Gamma$ and open neighbourhood $U$ of $x$ there exists
$y\in U$ such that $x,y$ are scrambled. Choose $a\in X^\Gamma$, then there exists $b\in X^\Gamma$ such that
$a,b$ are scrambled and (1) is valid.
\end{proof}
\section*{Acknowledgement}
\noindent The authors would like to express their deep thanks to the referee 
for his/her useful comments.

\[\underline{\SP\SP\SP\SP\SP\SP\SP\SP\SP\SP\SP\SP\SP\SP\SP\SP}\]
\noindent {\small {\bf Fatemah Ayatollah Zadeh Shirazi},
Faculty of Mathematics, Statistics and Computer Science,
College of Science, University of Tehran,
Enghelab Ave., Tehran, Iran
\linebreak
({\it e-mail}: f.a.z.shirazi@ut.ac.ir, fatemah@khayam.ut.ac.ir)}
\\
{\small {\bf Fatemeh Ebrahimifar},
Faculty of Mathematics, Statistics and Computer Science,
College of Science, University of Tehran,
Enghelab Ave., Tehran, Iran
\linebreak
({\it e-mail}: ebrahimifar64@ut.ac.ir)}
\\
{\small {\bf Maryam Hagh Jooyan},
Department of Mathematics, Tarbbiat Modares University,
Tehran, Iran
({\it e-mail}: haghjooyanmaryam@gmail.com)}
\\
{\small {\bf Arezoo Hosseini},
Faculty of Mathematics, College of Science, Farhangian University, Pardis Nasibe--shahid sherafat, Enghelab Ave., Tehran, Iran
({\it e-mail}: a.hosseini@cfu.ac.ir)}
\end{document}